\documentclass[12pt, a4paper, reqno]{amsart}

\usepackage{amsmath, amssymb, amsthm}
\usepackage[T1]{fontenc}
\usepackage[utf8]{inputenc}
\usepackage[all]{xy}
\usepackage{fullpage}
\usepackage[colorlinks=true,citecolor=cyan,linkcolor=magenta, urlcolor=blue, filecolor=green, backref=page]{hyperref}
\usepackage{hyperref}
\usepackage{indentfirst}

\newcommand{\bF}{\mathbb{F}}

\newcommand{\bQ}{\mathbb{Q}}
\newcommand{\bR}{\mathbb{R}}

\newcommand{\bZ}{\mathbb{Z}}

\newcommand{\fa}{\mathfrak{a}}

\DeclareMathOperator{\Gal}{Gal}

\DeclareMathOperator{\Nm}{Nm}

\DeclareMathOperator{\Sel}{Sel}

\DeclareMathOperator{\im}{im}

\DeclareMathOperator{\rk}{rank}

\DeclareMathOperator{\tor}{tor}

\theoremstyle{plain}
\newtheorem{theorem}{Theorem}[section]
\newtheorem{lemma}[theorem]{Lemma}
\newtheorem{proposition}[theorem]{Proposition}

\theoremstyle{definition} 

\theoremstyle{remark} 

\DeclareFontFamily{U}{wncy}{}
\DeclareFontShape{U}{wncy}{m}{n}{<->wncyr10}{}
\DeclareSymbolFont{mcy}{U}{wncy}{m}{n}
\DeclareMathSymbol{\Sha}{\mathord}{mcy}{"58}

\newcommand{\lara}[1]{\langle #1 \rangle}

\newcommand{\lcrc}[1]{\{ #1 \}}

\title{Infinitely Many Elliptic Curves of Rank Exactly Two II}

\author{Keunyoung Jeong}
\address{Department of Mathematical Sciences, Ulsan National Institute of Science and Technology, UNIST-gil 50, Ulsan 44919, Korea}
\email{kyjeongg@gmail.com}

\begin{document}

\maketitle

\begin{abstract}
Under the parity conjecture, an infinite family of elliptic curves of rank $2$ with
a torsion subgroup of order $2$ or $3$ is constructed.
\end{abstract}

\section{Introduction}


There are numerous results on the construction of an infinite family of
elliptic curves of rank at least $r$ and given torsion subgroups.
For example, Dujella and Peral \cite{DP} proved that there are infinitely many 
elliptic curves $E/\bQ$ such that 
$$ \left\{ \begin{array}{ccc}
\rk_{\bZ}(E(\bQ)) \geq 3, & E(\bQ)_{\tor} = \bZ/2\bZ \times \bZ/6\bZ, \\
\rk_{\bZ}(E(\bQ)) \geq 3, & E(\bQ)_{\tor} = \bZ/8\bZ.
\end{array}\right.$$
For the other torsion groups, the analogous results are listed in \cite{Duj}.

However, less is known regarding the construction of an infinite family of elliptic curves
over the rational numbers whose rank is \emph{exactly} $r$. 
The only known cases are $r = 0$ and $1$.
We recall the parity conjecture for elliptic curves over the rationals: For any elliptic curve $E/\bQ$,
$$ \textrm{ord}_{s=1}L(s,E) \equiv \rk_{\bZ}(E(\bQ)) \pmod{2}.$$ 

The author and Byeon \cite{BJ2} constructed an infinite family of elliptic curves 
over the rationals whose Mordell--Weil group is exactly $\bZ \times \bZ$.
In this study, we will prove the analogous results for other torsion subgroups, namely,
$\bZ/2\bZ$ and $\bZ/3\bZ$.

\begin{theorem} \label{main}
Under the parity conjecture, there are infinitely many elliptic curves $E$ 
such that $E(\bQ) \cong \bZ \times \bZ \times T$ for
$T = \bZ/2\bZ$, $\bZ/3\bZ$.
\end{theorem}

Let $p, q$ be prime numbers, $w_{E}$ be the root number of the elliptic curve $E$,
$E_{m}$ be an elliptic curve defined by the equation $y^2 = x^3 + mx$, and
$A_m$ be an elliptic curve defined by the equation $y^2 = x^3 + m^2$.
We will construct a family of elliptic curve $E_{(-pq)}$
and $A_{pq}$ whose root numbers are $+1$ and have
a nontrivial rational point using the following lemma:

\begin{lemma}[{\cite[Lemma 2.2]{BJ1}}] \label{mainlem}
Let $f(x) \in \mathbb{Z}[x]$ be a polynomial of degree $k$ and positive
leading coefficient. Let $A, B$ be relatively prime
odd integers, $g$ be an integer, and $i$, $j$ be positive integers with
 $0 < i, j < g$ and
$(i,g)=(j,g)=1$. We assume that there is at least one integer $m$ such that
$$
2f(m) \equiv Ai + Bj \pmod{g}\,\,\,\mbox{and}\,\,\, (AB, 2f(m))=1.
$$
Let $\mathcal{E}_{k}^{ABij}(N,f)$ be the number of positive integers
$n \in [1, N]$ with $2f(n) \equiv Ai + Bj \pmod{g}$ and $(AB,
2f(n))=1$ for which the equation $2f(n) = Ap_1 + Bp_2$ has no
solution in primes $p_1 \equiv i, p_2 \equiv j \pmod{g}$. Then there
is an absolute constant $c>0$ such that
$$
\mathcal{E}_{k}^{ABij}(N,f) \ll_{f} N^{1-\frac{c}{k}}.
$$
Thus, there are infinitely many integers $n$ such that
$$2f(n) = Ap_1 +Bp_2,$$
for some primes $p_1 \equiv i$ and $p_2 \equiv j \pmod{g}$.
\end{lemma}

That is, infinitely many elliptic curves 
$E_{-pq}$ and $A_{pq}$ will be constructed such that
\begin{equation} \label{proc1}
\left\{ \begin{array}{ccc}
\bZ \times \bZ/2\bZ \leq E_{(-pq)}(\bQ), & w_{E_{(-pq)}} = +1,\\
\bZ \times \bZ/3\bZ \leq A_{pq}(\bQ), & w_{A_{pq}} = +1.
\end{array} \right.
\end{equation} 
Subsequently, the upper bound of size of Selmer groups of $E_{(-pq)}$ and $A_{pq}$ will be calculated.
The size of the Selmer groups of $E_p$ and $A_p$ is determined by
the residue class of $p$ modulo $16$ and $9$, respectively
(see \cite[Proposition X.6.2]{Si1}, and  \cite[Corollary 7.7]{CP}). 
In the case of $E_{(-pq)}$ and $A_{pq}$, the Selmer groups are not determined only by
the residue classes of $p$ and $q$ modulo $16$ and $9$. However it will be shown that the upper bound of size of Selmer groups can be calculated in certain cases
(see Proposition \ref{2des}, \ref{3delta}). 
Combining these with (\ref{proc1}), we have Theorem \ref{main}.

\section{2-Torsion case}

We recall that an elliptic curve $E_m$ is defined by the equation
$y^2 = x^3 + mx$, where $m \in \bQ$. The torsion subgroup of $E_m(\bQ)$ is 
$\bZ/2\bZ$ when $m \neq 4$ and $-m$ is not square \cite[Proposition X.6.1]{Si1}.

\begin{lemma} \label{2rtintpt} (i)
If $m$ is not divisible by any square of integers, then
$$w_{E_m} = w_\infty w_2,$$
where $w_\infty = \mathrm{sgn}(m),$ whereas
$w_2 = -1$ if $m \equiv 1, 3, 11, 13 \pmod{16}$, and $w_2 = +1$ otherwise.\\
(ii) Let $a, b$ be nonzero integers. Then,
the elliptic curve $E_{b^2(a^2 - b^2)} : y^2 = x^3 + b^2(a^2 - b^2)x$
has a nontrivial integral point $(b^2, \pm ab^2)$.
\end{lemma}
\begin{proof}
(i) It follows by \cite[(10), (13)]{BS}, and (ii) can be verified by a direct calculation.
\end{proof}

We recall the concept of descent via two-isogeny \cite[Theorem X.4.9]{Si1}.
Let $M_K^0$ and $M_{K}^\infty$ be 
the set of infinite places and finite places of a number field $K$,
$E_m'$ be an elliptic curve
defined by the equation $y^2 = x^3 - 4mx$, $\phi : E_m \to E_m'$ be a
$2$-isogeny defined by
$$\phi(x,y) \longrightarrow (\frac{y^2}{x^2}, \frac{y(m-x^2)}{x^2}), $$
and $\phi'$ be its dual isogeny. Then, for 
$S= M_\bQ^\infty \cup \lcrc{v \in M_\bQ^0 : v \mid 2m}$, we have
$$\Sel_{\phi}(E/\bQ) \subset H^1(\bQ, E[\phi], S), 
$$
where $H^1(\bQ, E[\phi], S) \subset H^1(\bQ, E[\phi])$ is the set of cocycles
unramified outside $S$. For
$$\bQ(S,2) := \lcrc{x \in \frac{\bQ^\times}{(\bQ^\times)^2} : \textrm{ord}_v(x)= 0
\textrm{ for all } v \not\in S}, $$
there is an isomorphism
$\iota : \bQ(S,2) \to H^1(\bQ, E[\phi], S)$ defined by
$\iota(d)(\sigma) := d^\sigma/d$ for all $\sigma \in \Gal(\overline{\bQ}/\bQ)$.
We note that $E[\phi] \cong \bZ/3\bZ$ as a $G_{\bQ}$-module.
Let $\textrm{WC}(E/\bQ)$ be the Weil--Ch\^{a}telet group of the elliptic curve $E/\bQ$.
Then there is a map
$$\bQ(S, 2) \overset{\iota}{\cong} H^1(\bQ, E[\phi], S) \to \textrm{WC}(E/\bQ), \quad
d \to C_d(w,z) : dw^2 = d^2 - 4mz^4, $$
and for $d \in \bQ(S, 2)$,  $\iota(d) \in \Sel_{\phi}(E/\bQ)$ if and only if 
the homogeneous space $C_d$ is locally trivial for all $p \in S$. That is,
$$  \lcrc{d \in \bQ(S,2): C_d(\bQ_p) \neq 
\varnothing \textrm{ for all $p \in S$}} \overset{\iota}{\cong} \Sel_{\phi}(E/\bQ). $$
We simply write $d \in \Sel_{\phi}(E/\bQ)$ for $\iota(d)  \in \Sel_{\phi}(E/\bQ)$, and
denote by $C'_d$ the homogeneous space of $E'_m$ for $d \in \bQ(S, 2)$.

\begin{proposition} \label{2des}
Let $E = E_{(-pq)}$ and $E' = E'_{(-pq)}$ for some primes $p$ and $q$. \\
(i) If $pq \not\equiv \pm 1 \pmod{8}$, then
$\bZ/2\bZ \leq\Sel_{\phi}(E/\bQ)\leq (\bZ/2\bZ)^2$.\\
(ii) If one of $p$ and $q$ is not equivalent to $1$ modulo $4$, then 
$\bZ/2\bZ \leq \Sel_{\phi'}(E'/\bQ) \leq (\bZ/2\bZ)^2.$
\end{proposition}
\begin{proof}
(i) By previous arguments, we know that 
$$\bQ(S,2) = \lcrc{\pm 1, \pm 2, \pm p, \pm q, \pm 2p, \pm 2q, \pm pq, \pm 2pq},
$$ and $C_d : dw^2 = d^2 + 4pqz^4.$
By \cite[Proposition X.4.9]{Si1}, we have $pq \in \Sel_\phi(E_{pq}/\bQ)$.
The negative $d \in \bQ(S, 2)$ is not in $\Sel_\phi(E/\bQ)$ because 
$C_d(\bR)$ is empty. 

Let $(W,Z)$ be a $\bQ_2$-point of $C_2 : w^2 = 2 + 2pqz^4$. We may 
assume that $W \in 2\bZ_2$ and $Z \in \bZ_2$.
If $pq \not\equiv \pm 1 \pmod{8}$, then 
$W^2 \equiv 2 + 2pqZ^4 \pmod{8}$ does not have a solution.
Hence, if $pq \not\equiv \pm 1 \pmod{8}$, then 
$2 \not\in \Sel_{\phi}(E_{pq}/\bQ)$.
Consequently,
$\lara{pq} \leq \Sel_{\phi}(E/\bQ) \leq \lcrc{1, p, q, pq, 2p, 2q}$ which proves
(i).

(ii) We note that the homogeneous space $C_d'$ is defined by the equation
$dw^2 = d^2 - pqz^4.$
As in (i), we have $-pq \in \Sel_{\phi'}(E'_{pq}/\bQ)$.
We consider $C'_{-1} : w^2 + 1 = pqz^4$, and 
let $(W,Z)$ be a $\bZ_p$-point of $C'_{-1}$.
As $W^2 + 1 \equiv 0 \pmod{p}$,
there is no $\bQ_p$-point in $C'_{-1}$ when $p \not\equiv 1 \pmod{4}$.
Similarly, if $q \not\equiv 1 \pmod{4}$, then $C'_{-1}(\bQ_q) = \varnothing$.
Hence,  $-1 \not\in \Sel_{\phi'}(E'/\bQ)$ if one of $p, q$ is not equivalent to $1$ modulo $4$.

We consider $C'_{-2} : 2w^2 + 4 = pqz^4$. 
We may assume that a $\bQ_2$-point $(Z,W)$ of $C'_{-2}$ satisfies $W \in \bZ_2$
 and $Z \in 2\bZ_2$.
As the equation $2W^2 + 4 \equiv 0 \pmod{16}$ does not have a solution,
$-2 \not\in\Sel_{\phi'}(E'/\bQ)$. Similarly,
$C'_2(\bQ_2)$ does not have a solution because 
$2W^2 - 4 \not\equiv 0 \pmod{16}$. Therefore, $2 \not\in\Sel_{\phi'}(E'/\bQ)$.

Consequently, if one of $p$ and $q$ is not equivalent to $1$ modulo $4$, 
$$\lara{-pq} \leq \Sel_{\phi'}(E'/\bQ) \leq \lcrc{1, \pm p, \pm q, -pq, \pm 2p, \pm 2q, \pm 2pq}. $$
Let $A = \lcrc{1, \pm p, \pm q, -pq, \pm 2p, \pm 2q, \pm 2pq}$. Then,
all the possible groups between $A$ and $\lcrc{1, pq}$ as sets 
have order bounded by $4$.
\end{proof}

\begin{theorem}
There are infinitely many elliptic curves $E$ such that $w_E = +1$ and
$$
\bZ\times\bZ/2\bZ \leq E(\bQ)\leq \bZ\times\bZ\times\bZ/2\bZ.$$
That is, under the parity conjecture, 
there are infinitely many elliptic curves
whose Mordell--Weil groups are exactly $\bZ\times\bZ\times\bZ/2\bZ$.
\end{theorem}
\begin{proof}
There is a natural $\bQ$-isomorphism between $E_{t^4s} \cong E_s$ for
$t, s \in \bQ$, which is defined by $(x,y) \to (\frac{x}{t^2}, \frac{y}{t^3})$.
By Lemma \ref{2rtintpt} (ii), 
$E_{b^4(a^2 - b^4)} \cong E_{(a^2 - b^4)}$
has an (non-torsion) integral point.
We use Lemma \ref{mainlem} with $A=B=1$, $g = 16$, $i = 15$, $j = 3$,
and $f(n) = 2n^2$. As $m = 1$ satisfies
$ 2m^2 \equiv i + j \pmod{16}$, 
there are infinitely many integers $b$ such that
$2b^2 = p+q$ and $p \equiv 15, q \equiv 3 \pmod{16}$. 
Then for $a = \frac{p-q}{2}$,
$$a^2-b^4 = (a + b^2)(a - b^2) = -pq.$$
The torsion subgroup of $E_{(-pq)}$ is $\bZ/2\bZ$. 
As $pq \not\equiv \pm 1 \pmod{8}$ and $p,q \equiv 3 \pmod{4}$,
$$ 2 + \rk_{\bZ}(E(\bQ)) \leq \dim_{\bF_2}(\Sel_{\phi}(E/\bQ)) + \dim_{\bF_2}(\Sel_{\phi'}(E'/\bQ)) \leq 4,$$
by  \cite[Proposition X.6.2]{Si1} and Proposition \ref{2des}.
Finally $w_{E_{-pq}} = +1$, by Lemma \ref{2rtintpt} (i).
\end{proof}

\section{3-Torsion case}

We recall that for $A_m : y^2= x^3 + m^2$, where $m \in \bQ$, 
if $m \neq 1$ is sixth-power-free integer, then the torsion subgroup of $A_m(\bQ)$ is $\bZ/3\bZ$ 
(see \cite[Exercise 10.19]{Si1}).
As in Section 2, we have the following lemma.

\begin{lemma}\label{3rtintpt}
(i) If $m$ is square-free and prime to $6$, then  
$w_{A_m} = w_3 \prod_{p \mid m} w_p$, where
$$
\left\{
\begin{array}{ccc}
w_3 = -1 & \textrm{if $m^2 \equiv -2 \pmod{9}$,} \\  
w_3 = +1 & \textrm{otherwise}.
\end{array} \right.
$$ $$
\left\{
\begin{array}{ccc}
w_p = -1 & \textrm{if $p\mid m,$ and $p \equiv 2 \pmod{3}$,} \\  
w_p = +1 & \textrm{otherwise.}
\end{array} \right.
$$
(ii) Let $a, b$ be nonzero integers. Then 
the elliptic curve $A_{a(a^2 - b^2)} : y^2 = x^3 + a^2(a^2 - b^2)^2$
has an nontrivial integral point $(-a^2 + b^2, \pm (ab^2 - b^3))$.
\end{lemma}
\begin{proof}
The first part can be easily deduced by \cite[\S 9, Theorem]{Liv}.
The second part can be verified by a direct calculation.
\end{proof}

We recall the concept of descent via $3$-isogeny \cite[Definition 1.3]{CP}.
Let $K=\bQ(\sqrt{-3})$, 
$A'_m$ be an elliptic curve defined by the equation $y^2 = x^3 - 27m^2$, 
$\phi : A_m \to A_m'$ be an isogeny defined by
$$\phi : (x,y) \longrightarrow (\frac{x^3 + 4m^2}{x^2}, \frac{y(x^3-8m^2)}{x^3}),$$
and $\phi'$ be its dual isogeny.
There are $3$-descent maps
$$
\frac{A_m(\bQ)}{\phi' A_m'(\bQ)} \overset{\alpha}{\longrightarrow} \bQ(S,3)
\,\, \textrm{and} \,\,
\frac{A_m'(\bQ)}{\phi A_m(\bQ)} \overset{\alpha'}{\longrightarrow} K(S,3),
$$
where $S = M_{(\cdot)}^\infty \cup \lcrc{v \in M_{(\cdot)}^0 : v \mid 6m}$
for $(\cdot) = K$ or $\bQ$. The map $\alpha$ is defined by
$$\alpha(O) = 1, \,\, \alpha(0,m) = \frac{1}{2m},\,\, \textrm{and } \,\,
\alpha(x,y) = y - m. $$ 
We note that $\alpha'$ is defined by $\alpha'(x,y) = y - 3m\sqrt{-3},$ and the images of 
$\alpha'$ are in $K_N(S,3) = \lcrc{\overline{u} \in K(S, 3) : 
\Nm_{K/\bQ}(u) \in (\bQ^\times)^3}$.

\begin{lemma} \label{3desD1}
Let $p,q \geq 5$ be primes, and $A_{pq} : y^2 = x^3 + p^2q^2$ be elliptic curves.\\
(i) For any $\overline{d} \in \bQ(S,3)$, let $d$ be the unique cube-free representative of 
$\overline{d}$, and $d = d_1^2d_2$ be the unique representation such that 
$d_i$ are square-free and coprime. Then, the solvability of the homogeneous space $C_d$ 
is equivalent to that of
\begin{equation} \label{hom1}
C_{d_1, d_2, \frac{2pq}{d_1d_2}} : d_1X^3 + d_2Y^3 + \frac{2pq}{d_1d_2}Z^3 = 0. 
\end{equation}
We will denote $C_{d_1, d_2, \frac{2pq}{d_1d_2}}$ by
$(d_1, d_2, \frac{2pq}{d_1d_2})$. Moreover, we have $\im \alpha \leq \lara{2,p,q}$.
\\
(ii) Let $u_1, u_2, u_3 \nmid 3$. The homogeneous space 
$C : u_1X^3 + u_2Y^3 + u_3Z^3 = 0$, which is denoted by 
$(u_1, u_2, u_3)$, has a $\bQ_3$-point if and only if 
$u_i \equiv \pm u_j \pmod{9}$ for some $i \neq j$.
\end{lemma}
\begin{proof}
As in Section 2, there is an exact sequence
$
A_{pq}(\bQ) \overset{\alpha}{\to} \bQ(S, 3) 
\to \textrm{WC}(A_{pq}/\bQ). 
$
Hence, $\overline{d} \in \im \alpha$ if and only if the homogeneous space
$C_d$ has a nontrivial rational solution.
By \cite[Theorem 3.1 (i), (iii)]{CP}, the homogeneous space $C_d$ is trivial in 
$\textrm{WC}(E/\bQ)$ if and only if (\ref{hom1}) has a nontrivial rational solution, and
if (\ref{hom1}) has a nontrivial rational solution, then $d_1d_2 \mid 2pq$. 
Hence, (i) follows, whereas (ii) is exactly \cite[Lemma 5.9 (i)]{CP}.
\end{proof}

\begin{lemma} \label{3desD3}
Let $p,q \geq 5$ be primes, $A_{pq} : y^2 = x^3 + p^2q^2$ be elliptic curves, and
$\tau$ be a unique nontrivial element in $\Gal(K/\bQ)$.  \\
(i) For $\overline{d} \in K_N(S, 3)$, there is an element $v = v_1 + v_2\sqrt{-3}$ 
such that 
$v_i \in \bQ$ and $d = v^2 \tau(v)$. The solvability of the homogeneous space $C'_d$
is equivalent to that of
\begin{equation} \label{homeqD3}
2v_2 X^3 - 6v_1Y^3 + \frac{6pq}{v_1^2 + 3v_2^2}Z^3 + 6v_1X^2Y
-18v_2XY^2 = 0.
\end{equation}
(ii) 
For $\overline{d} \in \im\alpha'$, there exists an ideal $\fa$, of
$O_K$ such that $dO_K = \fa^2\tau(\fa)$ and $\Nm_{K/\bQ}(\fa)$ 
is a cubefree divisor of $2pq$ divisible only by primes 
that are split in $K/\bQ$.
\\
(iii) 
The homogeneous space defined by (\ref{homeqD3}) has a $\bQ_2$-point
if and only if the class $\tau(v)/v$ is a cube in  $\bF_{2^2}$.
\end{lemma}
\begin{proof}
\cite[Proposition 4.1. (1), Corollary 4.3, Lemma 6.4 (2)]{CP}, respectively.
\end{proof}

\begin{proposition} \label{3delta}
Let $p,q \geq 5$ be primes, and $A_{pq} : y^2 = x^3 + p^2q^2$ be elliptic curves. \\
(i) If $p, q \equiv \pm 2 \pmod{9}$, then 
$\bZ/3\bZ \leq \im\alpha \leq (\bZ/3\bZ)^2$. \\
(ii) If $p \equiv 2 \pmod{3}$ and $q \equiv 1 \pmod{3}$, then 
$0 \leq \im\alpha' \leq \bZ/3\bZ$.
\end{proposition}
\begin{proof}
(i)
By Lemma \ref{3desD1} (i), it suffices to consider the homogeneous space
$(d_1, d_2, \frac{2pq}{d_1d_2})$ for $d = d_1^2d_2$ 
such that $d_i$ are square-free, coprime, and  $d_1d_2$ divides $2pq$. 
As $\im\alpha$ is a group, $(d_1,d_2,\frac{2pq}{d_1d_2})$ is locally trivial 
if and only if $(d_2,d_1,\frac{2pq}{d_1d_2})$ is.
By exchanging $d_1$ and $d_2$,
there are only 14 choices for $(d_1, d_2, \frac{2pq}{d_1d_2})$.
Among them, $(1,1,2pq)$ and $(1,2pq,1)$ have 
a trivial solution, namely, $[1,-1,0]$ and $[0,1,-1]$, respectively.
Hence, $1, 2pq \in \im \alpha$. 
There are $4$-families, namely, 
\begin{align*}
\lcrc{(1,2,pq), (pq,2,1), (1,pq,2)},\\
\lcrc{(1,q,2p), (1,2p,q), (2p,q,1)}, \\
\lcrc{(2,p,q), (2,q,p), (p,q,2)},\\
 \lcrc{(2q,p,1), (1,2q,p), (1,p,2q)}.
\end{align*}
If one of the families is in $\im \alpha$, then all elements in the family are in $ \im \alpha$ because $2pq \in  \im \alpha$. Hence, it suffices to check one homogeneous space for each family.

By Lemma \ref{3desD1} (ii),
the homogeneous space $(1,2,pq)$ has a $\bQ_3$-solution if and only if $pq \equiv \pm 1, \pm 2 \pmod{9}$. Hence, $4 \not\in \im \alpha$ if 
$pq \not\equiv \pm1, \pm 2 \pmod{9}$. 
Similarly, we can show the following:
\begin{itemize}
\item $q^2$ does not lie in $\im \alpha$ when $q \not \equiv \pm 1$,
$p \not\equiv \pm 5$, and $q \not\equiv \pm2p \pmod{9}$,
\item $2p^2$ does not lie in $\im \alpha$ when $p \not \equiv  \pm2$, $q \not\equiv \pm2$, and $p\not\equiv \pm q \pmod{9}$,
\item $p^2$ does not lie in $\im \alpha$ when $p \not \equiv  \pm1$, 
$q \not\equiv \pm5$, and $p\not\equiv \pm2q \pmod{9}$.
\end{itemize}
If $p, q \equiv \pm 2 \pmod{9}$, then $4, p^2, q^2$ 
do not lie in $\im \alpha$.
Therefore, $\bZ/3\bZ \leq \im \alpha \leq (\bZ/3\bZ)^2$. 

(ii) By Lemma \ref{3desD3} (ii), if $d \in \im\alpha'$, then there exists 
$a$ such that $d = \zeta_3^ia^2\tau(a)$ and 
$\Nm_{K/\bQ}(a) \mid 2pq$ is divisible only by primes that split in $K/\bQ$.
In this case, $\Nm_{K/\bQ}(a) \mid q$. Therefore, 
$\im\alpha' \leq \lara{\zeta_3, q'^2\overline{q'}}$, where $q'$ is a prime of $K$ satisfying
$\Nm_{K/\bQ}(q') = q$. 

We consider $d = \zeta_3$. For $v = \zeta_3$, we have $\zeta_3 = v^2\tau(v)$ and
$\tau(v)/v \neq 1$ in $\bF_{2^2}$.
Therefore, the homogeneous space $C'_{\zeta_3}$
does not have a solution in $\bQ_2$ by Lemma \ref{3desD3} (iii).
Consequently, when $p\equiv 1$ and $q \equiv 2 \pmod{3}$, 
$\im\alpha'  \leq \bZ/3\bZ$.
\end{proof}

\begin{theorem}
There are infinitely many elliptic curves $E$ such that $w_E = +1$ and
$$
\bZ\times\bZ/3\bZ \leq E(\bQ)\leq \bZ\times\bZ\times\bZ/3\bZ.
$$
That is, under the parity conjecture, there are infinitely many elliptic curves
whose Mordell--Weil groups are exactly $\bZ\times\bZ\times\bZ/3\bZ$.
\end{theorem}
\begin{proof}
By Lemma \ref{3rtintpt}, the elliptic curve $A_{a^3(a^6-b^2)}$ 
has a non-torsion integral point. 
We use Lemma \ref{mainlem} with $A = 27, B=1, i=2, j=7$, and 
$f(n) = 2n^3$. As $2m^3 \equiv 27i + j \pmod{9}$ has a solution $m=-1$,
there are infinitely many integers $a$ such that $a^3 = \frac{27p+q}{2}$,
and $p \equiv 2$, and $q \equiv 7 \pmod{9}$. Then for $b = \frac{27p-q}{2}$,
$$(a^6 - b^2) =  (a^3 + b)(a^3 - b) = 27pq.$$
Therefore, there are infinitely many elliptic curves 
$A_{3^3pq} \cong A_{pq}$ with
at least one non-torsion point and $A_{pq}(\bQ)_{\textrm{tor}}=\bZ/3\bZ$.
By Proposition \ref{3delta}, $|\im \alpha| \leq 3^2$ and $|\im \alpha'| \leq 3$
if we choose $p \equiv 2$ and $q \equiv 7 \pmod{9}$.
By \cite[Proposition 2.2]{CP}, $|\im \alpha| |\im \alpha'| = 3^{\rk_\bZ A_{pq}(\bQ)+1}$.
Hence, $1 \leq \rk(A_{pq}(\bQ)) \leq 2$, and $w_{A_{pq}} = +1$ 
by Lemma \ref{3rtintpt}.
\end{proof}

\end{document}